\documentclass[10pt]{amsart}

\usepackage[all]{xy}                        %

\CompileMatrices                            

\UseTips                                    

\input xypic
\usepackage[bookmarks=true]{hyperref}       

\usepackage{amssymb,latexsym,amsmath,amscd}
\usepackage{xspace}
\usepackage{enumerate}
\usepackage{graphicx}
\usepackage{vmargin}
\usepackage{todonotes}

\DeclareMathOperator{\Res}{Res}

\reversemarginpar

\theoremstyle{plain}
\newtheorem{theorem}{Theorem}[section]
\newtheorem*{theorem*}{Theorem}
\newtheorem{proposition}[theorem]{Proposition}
\newtheorem{corollary}[theorem]{Corollary}
\newtheorem{lemma}[theorem]{Lemma}

\theoremstyle{definition}
\newtheorem{definition}[theorem]{Definition}

\newtheorem{remark}[theorem]{Remark}

\newcommand{\enm}[1]{\ensuremath{#1}}          %

\newcommand{\cal}[1]{\mathcal{#1}}

\renewcommand{\bar}[1]{\overline{#1}}

\newcommand{\CC}{\enm{\mathbb{C}}}

\newcommand{\NN}{\enm{\mathbb{N}}}
\newcommand{\RR}{\enm{\mathbb{R}}}

\newcommand{\ZZ}{\enm{\mathbb{Z}}}

\newcommand{\PP}{\enm{\mathbb{P}}}
\newcommand{\LL}{\enm{\mathbb{L}}}

\newcommand{\Bb}{\enm{\cal{B}}}

\newcommand{\Ee}{\enm{\cal{E}}}

\newcommand{\Ii}{\enm{\cal{I}}}

\newcommand{\Kk}{\enm{\cal{K}}}

\newcommand{\Oo}{\enm{\cal{O}}}

\newcommand{\Ss}{\enm{\cal{S}}}

\renewcommand{\phi}{\varphi}
\renewcommand{\theta}{\vartheta}
\renewcommand{\epsilon}{\varepsilon}

\begin{document}

\title[Twistor lines on algebraic surfaces]{Twistor lines on algebraic surfaces}

\author[A. Altavilla]{A. Altavilla${}^{\ddagger}$}\address{Altavilla Amedeo: Dipartimento Di Matematica, Universit\`a di Roma ``Tor Vergata", Via Della Ricerca Scientifica 1, 00133, Roma, Italy} \email{altavilla@mat.uniroma2.it}

\author[E. Ballico]{E. Ballico${}^{\dagger}$}\address{} \email{}\address{Edoardo Ballico: Dipartimento Di Matematica, Universit\`a di Trento, Via Sommarive 14, 38123, Povo, Trento, Italy} \email{edoardo.ballico@unitn.it}

\thanks{${}^{\dagger,\ddagger}$GNSAGA of INdAM;
 ${}^{\dagger}$MIUR PRIN 2015 ``Geometria delle variet\`a algebriche'';
 ${}^{\ddagger}$FIRB 2012 {\sl Geometria differenziale e teoria geometrica delle funzioni}, SIR grant {\sl ``NEWHOLITE - New methods in holomorphic iteration''} n. RBSI14CFME and SIR grant {\sl AnHyC - Analytic aspects in complex and hypercomplex geometry} n. RBSI14DYEB. The first author wishes to thank also the Clifford Research Group at Ghent University where part of this project was carried out.}

\date{\today }

\subjclass[2010]{Primary 14D21, 53C28; secondary 32L25}
\keywords{Twistor Fibration; Lines on Surfaces; Pl\"ucker Quadric}

\begin{abstract} 

We give quantitative and qualitative results on the family of surfaces in $\mathbb{CP}^3$ containing finitely many twistor lines. 
We start by analyzing the ideal sheaf of a finite set of disjoint lines $E$. We prove that its general element is a smooth surface containing $E$ and no other line. Afterwards we prove that twistor lines are Zariski dense in the Grassmannian $Gr(2,4)$. Then, for any degree $d\ge 4$, we give lower bounds on the maximum number of twistor lines
contained in a degree $d$ surface. The smooth and singular cases are studied as well as the $j$-invariant one.

\end{abstract}
\maketitle

\section{Introduction and Main Results}

Given a four dimensional Riemannian manifold $(M^{4},g)$, its twistor space  $Z(M)$ is the total space of a  bundle 
containing all the complex structures that can be defined on $M$ and are compatible with $g$. If $(M,g)$ is anti-self-dual, then $Z(M)$
is a complex manifold of (complex) dimension $3$. 
Moreover, a complex 3-manifold $Z$ is the twistor space of some 4-dimensional Riemannian manifold $M$
if and only if it admits a fixed-point-free anti-holomorphic involution $j:Z\to Z$ and a foliation by $j$-invariant
rational curves $\mathbb{CP}^{1}$ each of which has normal bundle $\Oo(1)\oplus \Oo(1)$ (see e.g.~\cite{lebrun}).

Since any complex structure on $M$ compatible with $g$ is as well compatible with any conformal
equivalent metric $e^{f}g$, the whole theory is invariant under conformal transformations of $M$.

The first interesting example is given by the 4-sphere $\mathbb{S}^{4}$, identified with the \textit{left} quaternionic
projective line $\mathbb{HP}^{1}$, whose twistor space is given by $\mathbb{CP}^{3}$ with fibers
$\mathbb{CP}^{1}$, i.e.:
$$
\mathbb{CP}^{1}\to\mathbb{CP}^{3}\stackrel{\pi}{\to}\mathbb{HP}^{1},
$$
where $\pi$ is the real analytic submersion defined by
$$
\pi[z_{0},z_{1},z_{2},z_{3}]=[z_{0}+z_{1}j,z_{2}+z_{3}j].
$$
The fibration $\pi$ is the main object of this paper and, from now on, we focus our attention only on it.
On the affine subset $\{ [p,q]\in\mathbb{HP}^{1}\,|\, p\neq 0\}$, if $q=q_{1}+q_{2}j$, we have that the fibers  of $\pi$ are explicitly given by 
$$
\begin{cases}
z_{2}=z_{0}q_{1}-z_{1}\bar q_{2}\\
z_{3}=z_{0}q_{2}+z_{1}\bar q_{1}.
\end{cases}
$$

In this case the fibers are identified with projective lines $l$ such that $j(l)=l$,
where $j:\mathbb{CP}^{3}\to\mathbb{CP}^{3}$ is the fixed-point-free anti-holomorphic involution given by
$$
j[z_{0},z_{1},z_{2},z_{3}]\mapsto [-\bar z_{1}, \bar z_{0}, -\bar z_{3},\bar z_{2}].
$$
Notice that the map $j$ coincides with the map induced, via $\pi^{-1}$, by quaternionic left multiplication by $j$ (see e.g.~\cite[Formula (5.8)]{gensalsto}).
We now state a formal definition for our main object.
\begin{definition}
A projective line $l\subset\CC\PP^{3}$ is said to be a \textit{twistor line} if $j(l)=l$, i.e.: if $l$ is a fibre for $\pi$.
\end{definition}

In this setting an important theme is the analysis of complex surfaces in $\mathbb{CP}^{3}$ transverse
to fibers. This because, any surface $\Ss\subset\CC\PP^{3}$, that can be considered a graph for $\pi^{-1}$,
produces a complex structure compatible with the standard round metric on $\pi(\Ss)\subset\mathbb{S}^{4}$.
In fact, the map $\pi$ restricted to any degree $d$ algebraic surface $Y$ can be considered as a 
$d$-fold branched covering over $\mathbb{S}^{4}$ and the set of twistor lines contained in $\Ss$ is contained in the
ramification locus. In particular, in this framework, the number of twistor lines contained in a surface plays an important role being an invariant under conformal transformations of $\mathbb{S}^{4}$ (for more details see, e.g.,~\cite{APS}).

Motivated by some recent result on the geometry of twistor lines in particular cases 
(degree 1~\cite{altavilla, shapiro}, degree 2~\cite{chirka,sv1}, degree 3~\cite{altavilla, altavillasarfatti, armstrong, APS, sv2} and degree 4~\cite{gensalsto}, see also~\cite[arXiv version v1]{sv1}), in this paper
we give some general algebro-geometric result related to them in the same spirit of the papers~\cite{ab,ballico}.

After the introduction, this paper has 2 sections. Section~\ref{s2} contains preparatory material
and standard tools from algebraic-geometry. The only
main result is Theorem~\ref{i1} below which regards surfaces containing a fixed set of lines.
Before state it we set the following standard notation: for any union of lines $E$ let $\Ii _E$ be the ideal sheaf of $E$
(see Section~\ref{s2} for definitions and properties). We denote by $|\Ii _E(d)|$ the projective space associated to the vector space $H^0(\Ii_E(d))$. 

\begin{remark}\label{a1}
Let $E$ be a union of $s$ disjoint lines.
Recall that, for $d\ge s >0$, the set $|\Ii _E(d)|$ is a
projective space of dimension $\binom{d+3}{3} - s(d+1)-1$. In fact, it is well known that for integers $t\ge s-1$ such that $t>0$ and $s\ge 0$, if $E\subset \CC\PP^3$ is a union of $s$ disjoint lines, then $h^1(\Ii
_E(t)) =0$ and $h^0(\Ii _E(t)) =\binom{t+3}{3} -s(t+1)$ (see e.g.~\cite[Corollary 1.10]{sidman}, where ``d-regular'' means exactly that  $h^1=0$).
\end{remark}

\begin{theorem}\label{i1}
Fix  integers $d\ge k >0$ with $d\ge 4$. Let $E\subset \CC\PP^3$ be a union of $k$ disjoint lines.  A general $Y\in |\Ii _E(d)|$ is a smooth degree $d$ hypersurface
containing $E$ and containing no line $L\subset \CC\PP^3$ with $L\cap E=\emptyset$.
\end{theorem}

The proof of the previous theorem passes through a number of lemmas analyzing the first two
cohomology numbers of $\Ii_{E}(d)$ and of $\Ii_{2q\cup E}(d)$, for some fat point $2q$.

In Section 3, we add the hypothesis that the set $E$ of Theorem~\ref{i1}
is composed by twistor lines. Therefore we have firstly to explain what a general set of twistor lines
is.
 Let $Gr(2,4)$ be the Grassmannian of lines in the complex projective space and $\Lambda \subset
Gr (2,4)$ denote the set of all twistor lines. Topologically we have $\Lambda \cong \mathbb{S}^{4}$.
In the beginning of the Section, by means of what we call   \textit{Density Lemma}~\ref{b2}, we give meaning to  the words ``$k$ general twistor lines'' for any integer $k>0$.
A first consequence of the Density Lemma is Corollary~\ref{c3.3} stating that, for any $s$ greater or equal to 5, there 
are $s$ twistor lines $L_{1},\dots, L_{s}$ such that no other line intersects all of them. We recall that in~\cite{APS} the 
authors give conditions in order to establish whether 5 twistor lines lie on a cubic. This property implies that 
there exist other two distinct lines intersecting all the given five twistor lines.

Now, before stating the main results of the third section, for any integer $d>0$ we set the following quantities
that will be our lower bounds:
\begin{itemize}
\item  $\nu(d):= \lfloor (\binom{d+3}{3}-1)/(d+1)\rfloor$;
\item  $\nu_n(d):= \nu(d-1)$ if $d\ge 2$ and $\nu_n (d):= 0$ if $d\le 1$;
\item  $\nu_s (d):= \nu(d-3)$ if $d\ge 4$ and $\nu_s (d):= 0$ if $d\le 3$;
\item $\nu_j (d):= \nu_n(d-8)=\nu (d-9)$ if $d \ge 9$ and $\nu_j (d):= 0$ if $d\le 8$. 
\end{itemize}
The 3 subscripts for $\nu$ stand for \textit{normal}, \textit{smooth} and $j$-\textit{invariant} as the following results will suggest.
More explicitly, 
 \begin{equation}\label{eqnu}
  \begin{cases}
  \nu(d) = (d^2+5d)/6\quad \mbox{if}\quad d\equiv 0,1\pmod{3}\\
\nu(d) =(d^2+5d+4)/6\quad \mbox{if}\quad d\equiv 2\pmod{3},
 \end{cases}
 \end{equation} 
and, if $d\ge 3$ we have 
 $$
 \begin{cases}
\nu_s (d) = (d-3)(d+2)/6\quad \mbox{if}\quad d\equiv 0,1\pmod{3}\\
\nu_s (d) = (d^2-d -  2)/6\quad \mbox{if}\quad d\equiv 2\pmod{3}. 
 \end{cases}
$$ 
Notice that, for $d\ge 2$, $\nu(d)<d^{2}$ and that $ \nu(d) \sim \nu_s (d) \sim \nu_n(d) \sim \nu_j (d) \sim d^2/6$ for $d\to \infty$.

%
%
%
%
%
%
%
%

Recalling that a surface is said to be \textit{integral} if it is reduced and irreducible, we can state our first main result.

\begin{theorem}\label{i2}
Fix an integer $d \ge 4$. 
\begin{enumerate}
\item There is an irreducible degree $d$ surface $Y\subset \CC\PP^3$ containing $\nu(d)$ general twistor lines.
\item There is an irreducible degree $d$ surface $Y\subset \CC\PP^3$ containing $\nu_n(d)$ twistor lines, none of them
intersecting
$\mathrm{Sing}(Y)$, and with $\mathrm{Sing}(Y)$ finite.
\item There is a smooth degree $d$ surface $Y\subset \CC\PP^3$ containing $\nu_s (d)$ general twistor lines.
\end{enumerate}
\end{theorem}

The three cases are proven separately. 
Firstly, we prove the existence of a complex projective space of dimension $\binom{d+3}{3} -\nu (d)(d+1)-1$ 
parametrizing a family of integral degree $d$ projective surfaces, all of them containing a general (but fixed) union 
$E\subset \CC\PP^3$ of $\nu(d)$ general twistor lines, namely $|\Ii_{E}(d)|$  (see also Corollary~\ref{b3} and Proposition~\ref{b5.1}).
Then, in the following step we prove that, reducing the number of lines from $\nu(d)$ to $\nu_{n}(d)$, then the 
general element $Y$ of $|\Ii_{E}(d)|$ is such that
$Sing(Y)$ is finite and does not intersect $E$.
 Afterwards, in the last step we prove that, if the number of twistor lines is reduced to $\nu_{s}(d)$, then the general
 element of $|\Ii_{E}(d)|$ is smooth.
 
Thanks to the Density Lemma, Theorem~\ref{i2} is true
for  twistor lines $L_1,\dots ,L_k$ belonging to  a  non-empty open subset of $\Lambda ^k$.

In  case (2) of Theorem~\ref{i2} we obviously allow the case $\mathrm{Sing}(Y) =\emptyset$. Since any surface $Y\subset \CC\PP^3$ is a locally
complete intersection, $\mathrm{Sing}(Y)$ is finite if and only if $Y$ is normal (\cite[II.8.22A]{h}).

The previous theorem can be specialized as follows.

\begin{theorem}\label{i2.1}
Fix integers $d \ge 10$ and $k$ such that $0\le k \le \nu_j (d)$. Then there is a smooth degree $d$ surface $Y\subset \CC\PP^3$
containing
$k$ general twistor lines and no other no other line L not intersecting any of the k twistor lines.
\end{theorem}


A surface $Y$ in $\CC\PP^{3}$ is said to be $j$-invariant, if $j(Y) = Y$. This concept
will be better explained at the end of Section~\ref{S3}. Examples of $j$-invariant surfaces are
the so-called \textit{real quadrics} in~\cite{sv1} and the smooth cubic analyzed in~\cite{APS,sv2}.
With the help of some observation on the topology of $j$-invariant surfaces in $|\Ii_{E}(d)|$
we are able to improve Theorem~\ref{i2.1} as follows.

\begin{theorem}\label{i3}  
Take $d$, $\nu_j (d)$ and $k\le \nu_j (d)$ as in Theorem~\ref{i2.1} with $d$ even; if $d\equiv 2\pmod{4}$ assume $k$ odd; if $d\equiv 0\pmod{4}$ assume $k$ even. Then there exists a degree $d$ smooth $j$-invariant surface $Y\subset \CC\PP^3$ such that $Y$ contains exactly $k$ twistor lines.
\end{theorem}

\begin{remark}
General results on the maximum number $N_{d}$ of lines on a smooth degree $d$ projective surface are given in~\cite{bs}. In this remark we recall some of them.
It is well known that any smooth cubic surface contains exactly 27 lines and, from the work of Segre~\cite{segre}, that $N_{4}=64$ and $N_{d}\leq (d-2)(11d-6)$. More recently it was proved in~\cite{chm} that $N_{d}\ge 3d^{2}$. Moreover,
thanks to~\cite{bs,chm}, we have $N_{6}\ge 180$, $N_{8}\ge 352$ $N_{12}\ge 864$ and $N_{20}\ge 1600$.

Since twistor lines are skew, it is interesting to look at the maximum number $S_{d}$ of skew lines contained in a smooth degree $d$ projective surface. When $d=3$
this number is exactly $S_{3}=6$. Thanks to~\cite{nikulin}, $S_{4}=16$. If $d\ge 4$, $S_{d}\le 2d(d-2)$ \cite{miyaoka}.
Concerning lower bounds, $S_{d}\ge d(d-2)+2$ ~\cite{rams}, but if $d\ge 7$ is an odd number then~\cite{bs}
$S_{d}\ge d(d-2)+4$.

In the particular case of twistor lines these number can be improved as follows: any cubic surface contains at most
5 twistor lines~\cite{APS}, while it was observed in~\cite[arXiv version v1]{sv1} that a smooth degree $d$ projective surface contains at most $d^{2}$ twistor lines and that there exists a quartic containing exactly 8 twistor lines. 
In Theorem~\ref{i2} case (3) we give a first general lower bound, for $d\ge 4$, on the maximal number of twistor lines 
$T_{d}$ lying on a smooth degree $d$ surface, obtaining $d^{2}/6\simeq\nu_{j}(d)\le T_{d}\le d^{2}$.
However, this bound is not optimal due to the quartic with 8 twistor lines in~\cite[arXiv version v1]{sv1}.
\end{remark}
In the next remark we give a possible interpretation in terms of OCS's of our results.

\begin{remark}
Thanks to Theorems~\ref{i2.1} and~\ref{i3}, for any general set of $k$ points 
$P=\{p_{1},\dots, p_{k}\}\subset\mathbb{S}^{4}$, there is an infinite number of conformally inequivalent OCS's,
induced by different smooth surfaces in $\mathbb{CP}^{3}$, that are singular at $P$ and cannot be extended at any point of 
$P$. Moreover, for any fixed degree $d\geq 10$, there are at least $\nu_{j}(d)$ different conformal classes of surfaces
of degree $d$. Even if this fact was already known, our results give new
information on the minimum degree $d_{\min}$ of a smooth surface $\mathcal{S}\subset\CC\PP^{3}$, such
that it contains $\{\pi^{-1}(p_{1}),\dots,\pi^{-1}(p_{k})\}$ and no other twistor line.
\end{remark}

\section{Preliminaries and Proof of Theorem~\ref{i1}}\label{s2}

In this section, by means of some technical lemma, we prove Theorem~\ref{i1}. In this part of 
the paper the geometry of the twistor projection is not involved and all the ingredients come from
basic geometric constructions (we use as main references the books~\cite{GH, h}). We recall the main tools.

For each closed subscheme
$A\subset \CC\PP^3$ let $\Ii _A \subseteq \Oo _{\CC\PP^3}$ denote its \textit{ideal sheaf}, i.e.:
$\Ii_{A}$ is defined to be the kernel of the morphism $\Oo_{\CC\PP^{3}}\to i_{*}\Oo_{A}$, $i$ being the inclusion map
(see \cite[p. 115]{h}).

If $A\subset F$, let $\Ii _{A,F} \subseteq \Oo _F$ be the ideal sheaf of $A$ in $F$.
For any closed subscheme $A\subset \CC\PP^3$ we denote by $\mathrm{Res}_F(A)$ the residual scheme of $A$ with respect to $F$, i.e. the closed subscheme
of $\CC\PP^3$ with the conductor $\Ii _A:\Ii _F$ as its ideal sheaf. We always have $\Res _F(A) \subseteq A$ and if $A$ is a reduced algebraic set, then $\Res _F(A)$ is the
closure in $\CC\PP^3$ of $A\setminus (A\cap F)$, i.e. the union of the irreducible components of $A$ not contained in $F$. Assuming $\deg (F)=f$, for any $t\in \ZZ$ we have a natural exact sequence of coherent sheaves on $\CC\PP^3$:
\begin{equation}\label{eqa1}
0 \to \Ii _{\mathrm{Res} _F(A)}(t-f) \to \Ii _A(t) \to \Ii _{A\cap F,F}(t) \to 0.
\end{equation}

For any scheme $A\subset F$ and every curve $C\subset F$ let $\mathrm{Res}_{C,F}(A)$ be the closed subscheme of $F$ with $\Ii
_{A,F}:\Ii _{C,F}$ as its ideal sheaf. If $F$ is smooth, then $C$ is an effective Cartier divisor of $F$. We get the following
residual exact sequence of coherent sheaves on $F$:
\begin{equation}\label{eqa2}
0 \to  \Ii _{\mathrm{Res}_{C,F}(A)}(t)(-C) \to \Ii _{A,F}(t) \to \Ii _{A\cap C,C}(t)\to 0.
\end{equation}
If $A $ is the disjoint union of two closed subschemes
of $\CC\PP^3$, say $A =A_1\sqcup A_2$ with each $A_i$ closed in $\CC\PP^3$, then $\Res _F(A) =\Res _F(A_1)\cup \Res _F(A_2)$ and $F\cap A =(F\cap A_1) \cup (F\cap A_2)$. Since $\Res _F(A_i)\subseteq A_i$,
we have $\Res _F(A_1)\cap \Res _F(A_2) =\emptyset$.

When dealing with singularities it is useful to exploit the so-called \textit{first infinitesimal neighborhood
of a point} (also called \textit{fat point}). We recall some of its features.
For any $q\in \CC\PP^3$ let $2q$ be the first infinitesimal neighborhood of $q$ in $\CC\PP^3$, i.e. the closed subscheme of $\CC\PP^3$ with $(\Ii _q)^2$ as its ideal sheaf. The scheme $2q$ is a zero-dimensional scheme with reduced scheme $(2q) _{\mathrm{red}} = \{q\}$ and if $q$ is contained in an affine or projective $n$-dimensional space, then $\deg(2q)=\binom{2+n-1}{n}$ (and hence, for $n=3,2,1$, we have $\deg(2q)=4,3,2$, respectively). If $q\notin F$, we have
$2q\cap F =\emptyset$ and $\mathrm{Res} _T(2q) = 2q$.
If $q\in F$ let $(2q,F)$ be the closed subscheme of $F$ with $(\Ii _{q,F})^2$ as its ideal sheaf. We have 
that $(2q,F) = 2q\cap F$ (scheme-theoretic intersection) and the scheme $(2q,F)$
is a zero-dimensional subscheme of $F$ with $(2q,F)_{\mathrm{red}} = \{q\}$ and if $F$ is smooth at $q$ then $\deg ((2q,F)) =3$ and $\mathrm{Res}_F(2q) = \{q\}$.

\begin{remark}\label{r2.1}
Take now as $F$ a plane $H\subset \CC\PP^3$ such that $q\in H$ and as $C$ a line $L\subset H$ such that $q\in L$ (and so with $f=1$). 
Hence $\deg ((2q,H)\cap L) =2$ and so $\Ii _{(2q,H)\cap L,L}(t)$ is a line
bundle on $L\cong \CC\PP^1$ with degree $t-2$. 
The cohomology of line bundles on $\CC\PP^1$ gives $h^1(L,\Ii _{(2q,H)\cap L,L}(t))=0$
for all $t> 0$. We have $\mathrm{Res}_{L,H}((2q,H)) =\{q\}$ and hence, recalling that $\Oo _H(t)(-L) = \Oo _H(t-1)$, 
the exact sequence~\eqref{eqa2}, can be written as,
\begin{equation}\label{eqabis}
0\to \Ii _{q,H}(t-1)\to \Ii _{(2q,H),H}(t)\to \Ii _{(2q,H)\cap L,L}(t)\to 0.
\end{equation}
We have $h^1(H,\Ii _{q,H}(x))=0$ for all $x\ge 0$ (see~\cite[Corollary 3.13 pag 150]{miranda}). 
Hence the long cohomology
exact sequence
of~\eqref{eqabis} gives $h^1(H,\Ii _{(2q,H),H}(t)) =0$ for all $t> 0$.

 Recall now that $\mathrm{Res}_H(2q) = \{q\}$ and $h^1(\Ii _q(t)) =0$ for all $t\ge 0$. Thus the exact sequence~\eqref{eqa1} for $F=H$ and $f=1$ 
 can be written as
 \begin{equation*}
0 \to \Ii _q(t-1) \to \Ii _{2q}(t) \to \Ii _{(2q,H),H}(t) \to 0
\end{equation*}
 and gives
 $h^1(\Ii _{2q}(t)) =0$ and $h^0(\Ii _{2q}(t)) = h^0(\Oo _{\CC\PP^3}(t)) -\deg (2q) = \binom{t+3}{3} -4$ for all $t>0$. The projective
 space $|\Ii _{2q}(t)|$ parametrizes all degree $t$ surfaces passing through $q$ and singular at $q$ (e.g. $|\Ii _{2q}(1)| =\emptyset$ and $|\Ii _{2q}(2)|$ is the set
 of all quadric cones with vertex containing $q$).
\end{remark}

The proof of Theorem~\ref{i1} relies on the computation of the first two cohomology numbers of 
$\Ii_{E}(t)$ and the analysis of such numbers in the presence of fat points.
The following results will enable us to perform such analysis.

In the following proposition, given a finite set of points $S$ in a plane $H$, we compute, for any 
$x\geq |S|-2$ the number $h^{1}(H,\Ii_{S,H}(x))$. It turns out that this is zero for $x\geq |S|-1$ while,
for $x=|S|-2$ is equal to 1 if and only if $S$ is contained in a line.

\begin{proposition}\label{a2}
Let $H\subset \CC\PP^3$ be a plane and $S\subset H$ a finite set. Set $s:= |S|$. 
\begin{enumerate}
\item We have $h^1(H,\Ii _{S,H}(t)) = 0$ for all $t\ge s-1$.
\item If $S\ne\emptyset$ and $S$ is not contained in a line, then $h^1(H,\Ii _{S,H}(s-2))=0$.
\item If $S\ne\emptyset$ and there is a line $L\subset H$ such that $S\subset L$, then $h^1(H,\Ii _{S,H}(s-2))=1$.
\end{enumerate}
\end{proposition}

\begin{proof}
Note that if $s\le 2$, then the set $S$ is always contained in a line (unique if and only if $s=2$), while for $s\ge 3$ the set $S$
is contained in at most one line. 

The proposition is true if
$s=0$, because
$h^1(H,\Oo _{H}(t))=0$ for all
$t$. It is also true for $s=1$, because $h^1(H,\Ii _{p,H}(t)) =0$ and $h^1(H,\Ii _{p,H}(-1)) =s=1$ for all $p\in H$ and all $t\ge 0$. 
Thus we may assume $s\ge 2$ and use induction on the integer $s$.


Assume that there is a line $L$ such that $S\subset L$. We have an exact sequence
\begin{equation}\label{eqa6}
0 \to \Oo _H(s-3)\to \Ii _{S,H}(s-2)\to \Ii _{S,L}(s-2)\to 0.
\end{equation}
Since $|S|=s$ the sheaf $\Ii _{S,L}(s-2)$ is the degree $-2$ line bundle on $L\cong \CC\PP^1$ and so $h^1(H,\Ii _{S,L}(s-2))=1$
(see~\cite[Corollary 3.13 pag 150]{miranda}). We have $h^1(H,\Oo _H(s-3))=0$ and, since $s=|S|>0$, then 
$h^{1}(L,\Ii _{S,L}(s-2))=1$ and
$h^2(H,\Oo _H(s-3)) =0$. To have that $h^1(H,\Ii _{S,H}(s-2))=1$, use the long cohomology exact sequence
of~\eqref{eqa6}.

We now pass to the case in which $S$ is not contained in any line.
Assume $s\geq3$. Fix $p\in S$ and set $A:= S\setminus \{p\}$. Fix a general line $L\subset H$
such that $p\in L$. We have $L\cap S=\{p\}$. Consider the residual exact sequence of $L$ in $H$
\begin{equation}\label{eqa4}
0\to \Ii _{A,H}(t-1)\to \Ii _{S,H}(t) \to \Ii _{p,L}(t)\to 0.
\end{equation}
Since $\Ii _{p,L}(t)$ is the degree $t-1$ line bundle on $L\cong \CC\PP^1$, we have $h^1(L,\Ii _{p,L}(x))=0$ for all
$x>0$. Hence, if $t>0$ and $h^1(\Ii _{A,H}(t-1)) =0$, then the cohomology exact sequence of~\eqref{eqa4} gives $h^1(\Ii _{S,H}(t)) =0$.
By the inductive assumption this is always the case if either $t\ge s-1$ or $t=s-2$ and $A$ is not contained in a line.
Therefore we may assume that $A$ is contained in a line $R$. We have $\{p\} \in S\setminus S\cap
R$ and hence we have a residual exact sequence
\begin{equation}\label{eqa7}
0 \to \Ii _{p,H}(s-3)\to \Ii _{S,H}(s-2)\to \Ii _{A,R}(s-2)\to 0.
\end{equation}
Since $s-3\ge 0$, we have $h^1(H,\Ii _{p,H}(s-3)) =0$. Since $\Ii _{A,R}(s-2)$ is the degree $-1$ line bundle on $R\cong
\CC\PP^1$, we have $h^1(R,\Ii _{A,R}(s-2))=0$ (see~\cite[Example IV.1.3.5]{h} or the Riemann-Roch formula~\cite[page 245]{GH} or~\cite[IV.1.3]{h}). 
To conclude the proof it is sufficient to use the
 long cohomology exact sequence of~\eqref{eqa7}.\end{proof}


%

The next two results are a consequence of Proposition~\ref{a2} in the case in which a point of $S$ is replaced with its first infinitesimal neighborhood. We need to recall the following remark.

\begin{remark}\label{a5}
Take any surface $F\subset \CC\PP^3$ and any zero-dimensional scheme $Z\subset F$. Since $Z$ is zero-dimensional, we have
$h^i(Z,\Ss )=0$ for every $i>0$ and every coherent sheaf $\Ss$ on $Z$. Thus the exact sequence
$$0 \to \Ii _{Z,F}(t)\to \Oo _F(t)\to \Oo _Z(t)\to 0$$gives $h^2(F,\Ii _{Z,F}(t))=h^2(F,\Oo _F(t))$. Moreover, we have $h^2(\Oo
_F(t))=0$ if and only if $t \ge -\deg (F)$.
\end{remark}

\begin{lemma}\label{a4}
Fix a plane $H\subset \CC\PP^3$, $p\in H$ and a set $S\subset H\setminus \{p\}$. Set $s:= |S|$ and $Z:= (2p,H)\cup S$.
\begin{enumerate}
\item We have $h^1(H,\Ii _{Z,H}(t)) =0$ for all $t\ge s+1$.
\item If $S\cup \{p\}$ is not contained in a line, then $h^1(H,\Ii _{Z,H}(s)) =0$.
\item If there is a line containing $S\cup\{p\}$, then  $h^1(H,\Ii _{Z,H}(s)) =1$.
\end{enumerate}
\end{lemma}

\begin{proof}
Since any two points of $H$ are collinear, if $s\le 1$, then $S\cup \{p\}$ is always contained in a line.

First assume the existence of a
line
$R\subset H$ such that
$S\cup \{p\}\subset R$. Note that
$Z\cap R$ is the disjoint union of $S$
and the degree $2$ scheme $(2p,R)$ with $p$ as its support. Hence $\deg (Z\cap R) =s+2$ and so $\Ii _{Z\cap R,R}(t)$ is a line
bundle on $R\cong \CC\PP^1$ with degree $t-s-2$. The cohomology of line bundles on $\CC\PP^1$ gives $h^1(R,\Ii _{Z\cap R,R}(t))=0$
if $t\ge s+1$ and $h^1(R,\Ii _{Z\cap R,R}(t))=1$ otherwise. We have $\mathrm{Res}_{R,H}(Z) =\{p\}$ and hence a residual exact
sequence
\begin{equation}\label{eqa8}
0\to \Ii _{p,H}(t-1)\to \Ii _{Z,H}(t)\to \Ii _{Z\cap R,R}(t)\to 0.
\end{equation}
We have $h^1(H,\Ii _{p,H}(x))=0$ for all $x\ge 0$ (see~\cite[Corollary 3.13 pag 150]{miranda}). 
Remark~\ref{a5} gives $h^2(H,\Ii _{p,H}(t-1))=0$. Hence the long cohomology
exact sequence
of~\eqref{eqa8} gives $h^1(H,\Ii _{Z,H}(t)) =0$ for all $t\ge s+1$ and $h^1(H,\Ii _{Z,H}(s))=1$.

Now assume that $S\cup \{p\}$ is not contained in any line and take a line $L\subset \CC\PP^3$ such that $p\in L$, $L$ contains
at least one point of $S$ and with $|L\cap S|$ minimal among all lines through $p$ intersecting $S$. Since there are at least
$2$ lines through $p$ meeting $S$, we have
$|L\cap S|
\le
\lfloor s/2\rfloor$. We get a residual exact sequence
\begin{equation}\label{eqa9}
0\to \Ii _{\mathrm{Res}_{L,H}(Z)}(t-1)\to \Ii _{Z,H}(t)\to \Ii _{Z\cap L,L}(t)\to 0.
\end{equation}The scheme $\mathrm{Res}_{L,H}(Z)$ is the union of $p$ and $S\setminus S\cap L$. Thus $\mathrm{Res}_{L,H}(Z)$
is a finite set with cardinality at most $\lfloor s/2\rfloor +1$. Proposition~\ref{a2} gives $h^1(H,\Ii
_{\mathrm{Res}_{L,H}(Z)}(t-1))=0$ for all $t\ge \lfloor s/2\rfloor +1$. Since $|S\cap L|\le s-1$, we have $\deg (L\cap Z) =2
+|L\cap S|
\le s+1$, the cohomology of line bundles on $R\cong \CC\PP^1$ gives $h^1(R,\Ii _{Z\cap L,L}(t))=0$ for all $t\ge s$. To conclude the proof, use the long
cohomology
exact sequence of~\eqref{eqa9}.
\end{proof}

We now pass to the case of lines.

\begin{lemma}\label{a6}
Fix an integer $k>0$. Let $E\subset \CC\PP^3$ be a disjoint union of $k$ lines. Fix $q\in \CC\PP^3\setminus E$.
Then $h^1(\Ii _{\{q\}\cup E}(t)) =0$ for all $t\ge k$.
\end{lemma}

\begin{proof}
Fix a line
$L\subseteq E$ and set
$A:= E\setminus L$. Let
$H$ be the plane spanned by $q$ and $L$. Since the lines of $E$ are pairwise disjoints, $H$ contains no line of $A$. Thus
$S:= A\cap H$ is a finite set of cardinality $k-1$ and $\mathrm{Res}_H(\{q\}\cup E) = A$. We have $H\cap (E\cup \{q\}) = S\cup \{q\}\cup L$ (as schemes) and hence we have
$h^1(H,\Ii _{(E\cup \{q\})\cap H}(t)) = h^1(H,\Ii _{S\cup \{q\}\cup L}(t)) =h^1(H,\Ii _{S\cup \{q\}}(t-1))$. By Proposition~\ref{a2} we have $h^1(H,\Ii _{\{q\}\cup S}(x)) =0$ for all $x\ge k-1$. Thus $h^1(H,\Ii _{(E\cup \{q\})\cap H}(t))=0$ for all $t\ge k$.

Consider now the following residual exact sequence
\begin{equation}\label{eqa10}
0 \to \Ii _A(t-1)\to \Ii _{\{q\}\cup E}(t) \to \Ii _{(\{q\}\cup E)\cap H}(t)\to 0.
\end{equation}
Remark~\ref{a1} gives $h^1(\Ii _A(x)) =0$ for all $x\ge k-2$, hence, thanks to the long cohomology exact sequence of~\eqref{eqa10} we get $h^1(\Ii _{\{q\}\cup E}(t))=0$ for all $t\ge k$.
\end{proof}

The following two lemmas deal with the interplay between a set of lines $E$ and a fat point.  They will be used in the proof of Theorem~\ref{i1} when dealing with singularities.

\begin{lemma}\label{a7}
Fix an integer $k>0$. Let $E\subset \CC\PP^3$ be a disjoint union of $k$ lines. Fix $q\in \CC\PP^3\setminus E$.
\begin{enumerate}
\item We have $h^1(\Ii _{2q\cup E}(t)) =0$ for all $t\ge k+1$.
\item If there is no line containing $q$ and intersecting all lines of $E$, then $h^1(\Ii _{2q\cup E}(k)) =0$.
\item If there is a line $R$ containing $q$ and intersecting all lines of $E$, then $h^1(\Ii _{2q\cup E}(k)) =1$.
\end{enumerate}
\end{lemma}

\begin{proof}
Fix a line $L\subseteq E$ and call $H$ the plane spanned by $L\cup \{q\}$. Set $A:= E\setminus L$ and $S:= A\cap H$.
We have $\mathrm{Res}_H(2q \cup E) = \{q\}\cup A$. By Lemma~\ref{a4} and~\ref{a6} we have $h^1(\Ii _{\{q\}\cup A}(t-1)) =0$ for all $t\ge k$
and if $t=k-1$ and there is no line $R$ through $q$ containing $q$ and meeting all lines of $E$, while $h^1(\Ii _{\{q\}\cup A}(k-2)) = 1$ if there is such $R$.
 We have $H\cap (2q\cup E) = L \cup S\cup (2q,H)$ and hence $h^1(H,\Ii _{H\cap (2q\cup E)}(t))= h^1(H,\Ii _{L\cup S\cup (2q,H),H}(t))
= h^1(H,\Ii _{S\cup (2q,H),H}(t-1))$. Since $|S|=k-1$, Lemma~\ref{a4} gives $h^1(H,\Ii _{S\cup (2q,H),H}(t-1))=0$ if either
$t>k$ or $t=k$ and there is no line through $q$ containing $A$ and $h^1(H,\Ii _{S\cup (2q,H),H}(k-1))=1$ if there
is a line $R\subset H$ containing $\{q\}\cup A$ and hence meeting each line of $A$; since $L, R\subset H$, we get $L\cap R\ne
\emptyset$. Thus the residual exact sequence of $H$:
\begin{equation}\label{eqa+a1}
0 \to \Ii _{\{q\}\cup A}(t-1)\to \Ii _{2q\cup E}(t)\to \Ii _{L\cup (2q,H),H}(t)\to 0
\end{equation}
gives the lemma, except that in the set-up of (3) it only gives  $h^1(\Ii_{2q\cup E}(k)) \le 1$. Assume then $h^1(\Ii _{2q\cup E}(k)) =0$, i.e. assume $h^0(\Ii _{2q\cup E}(k)) =h^0(\Ii _{\{q\} \cup E}(k))-3$. Since $\deg (R\cap (\{q\}\cup E))=k+1$, the line $R$ is in the base locus of $|\Ii _{\{q\}\cup E}(k)|$. Thus the degree $2$ scheme $R\cap 2q$ is in the base locus of $|\Ii _{\{q\}\cup E}(k)|$ and hence $h^0(\Ii _{2q\cup E}(k)) \ge h^0(\Ii _{\{q\}\cup E}(k))-2$.
\end{proof}

We now examine the case in which $q\in E$.
\begin{lemma}\label{a8}
Fix an integer $k>0$, a union $E\subset \CC\PP^3$ of $k$ disjoint lines and $q\in E$. Then $h^1(\Ii _{2q\cup E}(x)) =0$ for
all $x\ge k$.
\end{lemma}

\begin{proof}
Call $L$ the line of $E$ containing $q$ and set $A:= E\setminus L$. Let $H\subset \CC\PP^3$ be a plane containing $L$. Set $S:=
A\cap H$. Since any two lines of $H$ meet while the lines of $E$ are disjoint, $S$ is a finite set, $|S|=k-1$ and $S\cap L=\emptyset$. Since $q\in H$, we have $\mathrm{Res}_H(2q) =\{q\}$. Since $L$ is the only irreducible component of $E$ contained in $H$, we have $\Res _H(E)=A$. It is easy to check (if you prefer use local coordinates around $q$) that
$\mathrm{Res}_A(2q\cup E) =\{q\}\cup A$. Since $A$ is a union of $k-1$ disjoint lines, Lemma~\ref{a6} gives
$h^1(\Ii _{\mathrm{Res}_A(2q\cup E)}(t-1)) =0$ for all $t\ge k$.

We have $H\cap (2q\cup E) = L \cup S \cup (2q,H)$.  Note that $L\cap (2q,H) = (2q,L)$. Thus $\{q\}$ is the residual scheme $\mathrm{Res}_L(L\cup (2q,H))$ of $H$
with respect to $L$ (seen as a divisor in $H$). Thus 
$$h^1(H,\Ii _{L\cup S\cup (2q,H)}(t)) = h^1(H,\Ii _{\{q\}\cup S}(t-1)).$$ 
Since $k = |\{q\}\cup S|$, Proposition \ref{a2} gives
$h^1(H,\Ii _{\{q\}\cup S}(t-1))=0$ for all $t\ge k$. To conclude it is sufficient to look at the long cohomology exact sequence of the residual exact sequence~\eqref{eqa+a1} of $H$ in $\CC\PP^3$.
\end{proof}

We are now in position to prove our first main result, but first we recall the following fact.

\begin{remark}\label{ii4}
The set $\Delta$ of all degree $d$ surfaces of $\CC\PP^3$ (allowing
also the reducible ones and those with multiple components), is a projective space of dimension $\binom{d+3}{3} -1$. The set
of all singular $Y\in \Delta$ is a non-empty hypersurface $\Sigma$ of $\Delta$. Every complex closed subset of positive
dimension of $\Delta$ meets $\Sigma$. Thus there is no compact complex family of positive dimension parametrizing different
smooth degree $d$ surfaces of $\CC\PP^3$. Allowing singularities (but not allowing decomposable surfaces), we may get
complex positive-dimensional compact families of solutions.  
\end{remark}

%

\begin{proof}[Proof of Theorem~\ref{i1}]

 To prove Theorem~\ref{i1} it is sufficient to prove that a general $Y\in |\Ii _E(d)|$
is smooth and contains no line $L$ with $L\cap E=\emptyset$. We use several times that a finite intersection of non-empty
Zariski open subsets of $|\Ii _E(d)|$ is non-empty and obviously open and hence Zariski-dense
in $|\Ii _E(d)|$ to check separately the smoothness condition and the condition on lines.

Fix any
line
$L\subset
\CC\PP^3$ such that
$L\cap E=\emptyset$. By Remark~\ref{a1} for the integer $s=k+1$ the set $|\Ii _{E\cup L}(d)|$ is a linear subspace of $|\Ii
_E(d)|$ with codimension
$d+1$.
Since the set $Gr(2,4)$ of all lines of $\CC\PP^3$ has dimension $4$, the set of all $Y\in |\Ii _E(d)|$ containing at least one
line $L$ with $L\cap E=\emptyset$ is contained in an algebraic subvariety of $|\Ii _E(d)|$ with codimension at least $d+1-4>0$
.
Thus a general $Y\in |\Ii _E(d)|$ contains no line $L$ with $L\cap E=\emptyset$.

We now deal with the smoothness property of a general $Y\in |\Ii _E(d)|$.
For any $q\in \CC\PP^3$ set $\Sigma _q:= |\Ii _{2q\cup E}|$, i.e. $\Sigma _q$ is the set of all $Y\in |\Ii _E(d)|$
singular at $q$. 
Thanks to Remark~\ref{ii4}, it is sufficient to prove that $\cup _{q\in \CC\PP^3} \Sigma _q$ is contained in a proper
closed algebraic subvariety of $|\Ii _E(d)|$. We have $\dim E=1$ and, for each $q\in E$, Lemma~\ref{a8} gives $\dim \Sigma _q
=\dim |\Ii _E(d)| -2$. Thus $\cup _{q\in E} \Sigma _q$ is contained in a proper closed subvariety of $|\Ii _E(d)|$. Let
$\Delta$ be the set of all
$q\in \CC\PP^3\setminus E$ such that $h^1(\Ii _{2q\cup E}(d)) >0$. If $q\in \CC\PP^3\setminus (\Delta \cup E)$, then $\dim \Sigma _q
= \dim |\Ii _E(d)| -4$. Since $\dim \CC\PP^3 =3$, the set $\cup _{q\in \CC\PP^3\setminus (\Delta \cup E)} \Sigma _q$ is contained in
a proper closed subvariety of
$|\Ii _E(d)|$. Now assume $q\in \Delta$. By Lemma~\ref{a7} we have $k=d$, $\dim \Sigma _q = \dim |\Ii _E(d)| -3$ and $\Delta$
is the intersection with $\CC\PP^3\setminus E$ of the union of the lines $R$ meeting each line of $E$. Since $\dim \Sigma _q =
\dim |\Ii _E(d)| -3$ for each $q\in \Delta$, it is sufficient to prove that $\dim \Delta \le 2$. Let $Q$ be any quadric
surface containing
$3$ of the lines of
$E$. Bezout theorem implies that each line $R$ meeting each line of $E$ is contained in $Q$. Thus $\Delta \subseteq Q$ and so we get the thesis.
\end{proof}

\section{Finitely Many Twistor Lines and Proof of Theorems~\ref{i2}, \ref{i2.1} and~\ref{i3}}\label{S3}
This section is devoted to the proof of Theorems~\ref{i2}, \ref{i2.1} and~\ref{i3}, therefore we 
start to deal with twistor lines. At first we prove the Density Lemma~\ref{b2} which allows us 
to prove a number of results on surfaces containing twistor lines only using already known material.
Then, we will use the map $j$ to prove the more precise Theorem~\ref{i3}.
We recall that the Grassmannian $Gr(2,4)$ can be identified with the Klein quadric $\Kk=\{t_{1}t_{6}-t_{2}t_{3}+t_{4}t_{5}=0\}\subset \CC\PP^{5}$ via Pl\"ucker embedding. The map $j$ induces, then, a map on $\CC\PP^{5}$ (that will be also denoted by $j$) defined as 
\begin{equation*}
j([t_1:t_2:t_3:t_4:t_5:t_6]) =[\bar{t}_1:\bar{t}_5:-\bar{t}_4:-\bar{t}_3:\bar{t}_2:\bar{t}_6],
\end{equation*}
(see~\cite{altavilla, gensalsto}).
This new map $j$ identifies twistor lines in $\Lambda$ and it is an anti-holomorphic involution.

Let $k$ be a positive integer, then we have $\Lambda ^k\subset Gr (2,4)^k$ and both sets are compact with their euclidean topology.
$Gr (2,4)^k$ also has the Zariski topology, which obviously, if $k\ge 2$, is strictly finer than the product topology.

We now introduce a formalism that allows us to state simultaneously a finite number of properties
that may be true or false for a Zariski open dense set of twistor fibers. In particular, thanks to this result we give meaning to the words ``general sets of twistor lines''.
First of all fix an integer $k>0$ and a property $\wp$ which may be true or false for any set of $k$ different twistor lines. If $k\ge 2$
we extend $\wp$ to all $\LL =(L_1,\dots ,L_k)\in \Lambda ^k$ saying that $\wp$ is false at $\LL$ if $L_m=L_n$ for some $m\ne n$, while, if $L_m\ne L_n$ for all $m\ne n$, $\wp$ is true at $\LL$ if and only if it is true at $\{L_1,\dots ,L_k\}$. 
\begin{definition}\label{b1}
Fix an integer $k>0$ and a property $\wp$ that is is either \textit{true} or \textit{false} for any $k$ distinct twistor lines. We say that
$\wp$ is true for $k$ \textit{general twistor lines} or that $k$ \textit{general twistor lines} satisfy (the property) $\wp$ or that $\wp$ is true for a general
union of $k$ twistor lines if there is a non-empty Zariski open subset
$U$ of
$Gr (2,4)^k$ such that
$\wp$ is true for all $\LL =(L_1,\dots ,L_k)\in U\cap \Lambda ^k$.
\end{definition}

Note that the union of the big diagonals of $Gr (2,4)^k$, $k\ge 2$, is Zariski closed in $Gr (2,4)^k$ and it is different from
$Gr (2,4)^k$. Thus in Definition~\ref{b1} it is not restrictive to assume that all $(R_1,\dots ,R_k)\in U$ have $R_m\ne R_n$
for all $m\ne n$. This is the main reason for our extension of any $\wp$ as false when the $k$ twistor lines are not distinct.

\begin{lemma}[Density Lemma]\label{b2}
For any $k>0$ the following statements hold true.

\quad (a)  $\Lambda ^k$ is dense in $Gr (2,4)^k$ in which the latter has
the Zariski topology. 

\quad (b) Fix finitely many properties $\wp _i$, $1\le i\le s$, about $k$ distinct twistor lines and which are true for
$k$ general twistor lines. Then there are $k$ distinct twistor lines for which all $\wp _i$, $1\le i\le s$, are true.
\end{lemma}

\begin{proof}
Part (a) is true if $k=1$ because, as said before, $\Lambda$ is the set fixed by $j$
and we recall that the fixed locus of an anti-holomorphic involution which acts on 
a connected smooth algebraic manifold $M$, is either empty or Zariski dense in $M$.

Now assume $k\ge 2$. Every non-empty Zariski open subset of $Gr (2,4)^k$ is dense in $Gr (2,4)^k$. The Zariski topology of $Gr (2,4)^k$ is finer than the product topology of its factors. Thus if $U_i$, $1\le i\le s$,  are non-empty
Zariski open subsets of $Gr (2,4)$, then $U_1\times \cdots \times U_k$ is a non-empty and dense Zariski open subset of $Gr (2,4)^k$. 

Since a finite intersection of non-empty Zariski open subsets of $Gr (2,4)^k$ is non-empty, open and dense in $Gr (2,4)^k$, we get the thesis of part (a) and, as a direct consequence, also part (b).
\end{proof}

Thanks to the previous lemma, we are able to prove the following two results. The first is interesting
by itself and states that for any number $s\ge 4$ there are $s$ twistor lines such that no other line
intersects all of them. In the second corollary we compute first two cohomology numbers of $\Ii_{E}(t)$, 
$E$ being a general union of twistor lines.

\begin{definition}
Consider a set of $k$ disjoint lines $L_{1},\dots, L_{k}$. These are said to be~\textit{collinear} if there is another
line $R$ intersecting all of them.
\end{definition}

Thanks to~\cite[Proposition 2]{APS}, if five lines are collinear, then they lie on a cubic surface and if five
skew lines lie on a cubic surface, then they are collinear. In particular, five twistor lines that lie on a cubic surface
are collinear.

\begin{corollary}\label{c3.3}
For any integer $s\ge 5$, there are $s$ twistor lines which are not collinear.
%
\end{corollary}

\begin{proof}
Of course, it is sufficient to do the case $s=5$. Take any $4$ distinct twistor lines $L_1,L_2,L_3,L_4$ not all of them contained
in a smooth quadric surface (e.g. take $4$ of the twistor lines contained in the Fermat cubic (see~\cite{sv2}). 
There is a unique
quadric surface $T$ containing $L_1,L_2,L_3$, this quadric surface is smooth and $L_1,L_2,L_3$ are in the same ruling of
$T$, say $L_i\in |\Oo _T(1,0)|$, $i=1,2,3$ (see~\cite[page 478]{GH}). By Bezout theorem each line $R\subset \CC\PP^3$ intersecting each
$L_i$, $i=1,2,3$,
is contained in $T$. Since any two lines in the same ruling of $T$ are disjoint, $R$ must be an element of $|\Oo _T(0,1)|$.  Since
$L_4\nsubseteq T$, the set
$L_4\cap T$ has either cardinality
$2$ (case
$L_4$ transversal to
$T$) or cardinality $1$ (case $L_4$ tangent to $T$). Set $e:= |T\cap L_4|$ and let $D_i$, $1\le i\le e$ be the line
in $|\Oo _T(0,1)|$ containing one of the points of $L_4\cap T$. These $e$ lines $D_i$ are the only lines of $\CC\PP^3$
meeting each of the lines $L_1,L_2,L_3,L_4$. It is sufficient to prove the existence of a twistor line
not meeting each of the $e$ lines $D_i$. Let $\Gamma _i$, $1\le i\le e$, denote the set of all lines of $\CC\PP^3$ meeting
$D_i$. Each $\Gamma _i$ is a hypersurface of $Gr(2,4)$
of degree $2$, but then, to conclude, it is sufficient to invoke the Density Lemma.
\end{proof}

\begin{corollary}\label{b3}
Fix an integer $k>0$. Let $E\subset \CC\PP^3$ be a general union of $k$ twistor lines. Then $h^0(\Ii _E(t)) =\max
\{0,\binom{t+3}{3} -k(t+1)\}$ and $h^1(\Ii _E(t)) = \max \{0, k(t+1)-\binom{t+3}{3}\}$ for all $t\in \NN$.

\end{corollary}

\begin{proof}
Note that for a fixed $t$ and any union $F$ of $k$ disjoint lines we have $h^0(\Ii _F(t)) =\max
\{0,\binom{t+3}{3} -k(t+1)\}$ $\Leftrightarrow$ $h^1(\Ii _F(t)) = \max \{0, k(t+1)-\binom{t+3}{3}\}$  $\Leftrightarrow$ $h^0(\Ii
_F(t))\cdot h^1(\Ii _F(t))=0$. By Remark~\ref{a1} we have $h^1(\Ii _F(t))=0$ for all $t\ge k-1$. Thus for a fixed $k$ we only need to check finitely many integers $t$, i.e. $0\leq t\leq k-2$.
For each $t$ the condition $h^0(\Ii
_F(t))\cdot h^1(\Ii _F(t))=0$ is an open condition for the Zariski topology of $Gr (2,4)$ by the semicontinuity theorem for
cohomology \cite[III.12.8]{h}. By~\cite[Theorem 0.1]{hh} or~\cite{ccg} these conditions are satisfied in a non-empty Zariski open subset
of $Gr (2,4)^k$. Using the Density Lemma~\ref{b2} we get the final thesis.
\end{proof}

We now pass to the proof of the first part of Theorem~\ref{i2}.

Let $E\subset \CC\PP^3$ be the union of $\nu(d)$ general lines. Notice that $\nu(d)$ is
the maximal integer
$y$ such that
$y(d+1)< \binom{d+3}{3}$. Therefore, by Corollary~\ref{b3} we have
 $h^1(\Ii _E(d))=0$ and thanks to Theorem~\ref{i1}, $\dim |\Ii _E(d)| =
d$ if $d\equiv 0,1\mod{3}$ and $\dim |\Ii _E(d)| =
(d-2)/3$ if $d\equiv 2\mod{3}$.  In particular, if $d=2$, then $\nu(d)=3$ and $\dim |\Ii _E(2)| =0$.
We start by proving the following general fact.

\begin{proposition}\label{b5.1}
Let $E\subset \CC\PP^3$ be the union of $\nu(d)$ general twistor lines.
If $d \ge 5$, then every $Y\in |\Ii _E(d)|$ is integral, i.e. it is reduced and irreducible. If $d\le 4$, then a general $Y\in |\Ii _E(d)|$ is integral.
\end{proposition}

\begin{proof}
Thanks to Corollary~\ref{b3}, we have that $h^0(\Ii _E(d-1)) =0$ therefore, since 
$E\subset Y_{\mathrm{red}}$, then each $Y\in |\Ii_{E}(d)|$ has no multiple component, i.e.: all $Y$ are reduced.

Assume now that $|\Ii_{E}(d)|\ni Y =Y_1\cup Y_2$ with $Y_1$ a surface of degree $y$, $Y_2$ a surface of degree $d-y$ and $0<y<d$ (we do
not assume
that some of the $Y_i$ are irreducible). Let $E_i$, $i=1,2$, be the set of all lines $L\subset E\cap Y_{i}$. 
We have $E= E_1\cup E_2$. Since $E$ is a union of general twistor lines, every union of
some of the lines of $E$ is a union of general twistors lines. Moreover if $x_i:=\deg (E_i)$, 
we have $x_1+x_2\ge \nu(d)$.

Set now $d\le 4$ and recall that $\nu(1)=1$, $\nu(2)=3$, $\nu(3)=4$, $\nu(4)=6$.
Thus (assuming $y\le d-y$), we only have to check $d=3$ and $y=1$ or $d=4$ and $y=2=d-y$.
In both cases $d-y=2$, hence, since $\dim |\Ii_{E_{2}}(2)|=0$ (and $\dim |\Ii_{L}(1)|=1$, for $L\in E$), then 
we have $|\Ii_{E_{2}}(2)|=\{Y_{2}\}$. But $\dim |\Ii_{E}(3)|=4$ and $\dim |\Ii_{E}(4)|=5$ and so a
general $Y\in |\Ii_{E}(d)|$, $d=3,4$, is irreducible.

Now assume $d\ge 5$. By Corollary~\ref{b3} we have
$$
 \binom{y+3}{3}>x_{1}(y+1), \qquad  \binom{d-y+3}{3}> x_{2}(d-y+1),
$$
that is $x_1\le \nu(y)$ and $x_2\le \nu(d-y)$. Thus $\nu(y) +\nu(d-y) \ge \nu(d)$. 
Recall from Equation~\eqref{eqnu}, that, for all $t\ge 4$ we have
$$\frac{t^2+5t}{6} \le \nu(t) \le \frac{t^2+5t+4}{6}.$$ 
Thus, 
$$
d^2+5d \le y^2+5y+4+(d-y)^2+5(d-y)+4, 
$$
that is $y(d-y)\le 4$. Since $1\le y\le d-1$ and $d\ge 5$, the only possibility is $d=5$ and either $y=1$ or $y=4$. 
But have $\nu(1) = 1$, $\nu(4) =6$ and $\nu(5) =9>\nu(1)+\nu(4)$, obtaining then a contradiction.
\end{proof}

We can now perform the following proof. 
\begin{proof}[Proof of Theorem~\ref{i2} case (1):]
Consider a general set $E$ of $\nu(d)$ twistor lines with $d\ge 4$. Thanks to Corollary~\ref{b3} we have $\dim |\Ii_{E}(d)|>0$. Therefore, Proposition~\ref{b5.1} implies that, for $d\ge 4$, there is a positive dimensional family 
of irreducible projective surfaces containing $\nu(d)$ lines.
\end{proof}

Before proving the second part of Theorem~\ref{i2}, we need a technical result. In particular,
the following lemma will be used to control the singular locus of a generic surface
$Y\in|\Ii_{E}(d)|$.

\begin{remark}\label{b4.0}
Fix integers $a\ge x\ge 0$. Let $S\subset \CC\PP^1$ be a finite set with $|S|=x$. Since $h^1(\Oo _{\CC\PP^1}(a-x)) =0$
and the line bundle $\Oo _{\CC\PP^1}(a-x)$ is base-point-free, $S$ is the scheme-theoretic base locus of the linear system
$|\Ii _S(a)|$ on $\CC\PP^1$.
\end{remark}

\begin{lemma}\label{b4}
Fix integers $d$ and $x$ such that $d\ge 2$ and $0\le x \le 1+\frac{d(d-1)}{2}$. Let $S \subset \CC\PP^2$ be a
general finite subset of $\CC\PP^2$ with $|S|=x$. Then $h^1(\Ii _S(d)) =0$ and $S$ is the scheme-theoretic base locus of $|\Ii _S(d)|$. 
\end{lemma}

\begin{proof}
By the semicontinuity theorem for cohomology it is sufficient to find a finite set $A\subset \CC\PP^2$ such that $|A|=x$,
 $h^1(\Ii _A(d)) =0$ and $A$ is the scheme-theoretic base locus of $|\Ii _A(d)|$. Fix $d-1$ general lines $R_1,\dots ,R_{d-1}$, a point
$p\in \CC\PP^2\setminus (R_1\cup \cdots \cup R_{d-1})$, and
sets $B_i\subset R_i$, $i=1,\dots ,d-1$ with $|B_i| = i$ for all $i$ and $B_i\cap R_j=\emptyset$ for all $j\ne i$.
Set $B_0:= \{p\}$ and $B:= \cup _{i=0}^{d-1} B_i$. We have $|B| =1+\sum _{i=1}^{d-1} i = 1+(d-1)d/2$. Let $R_0 \subset \CC\PP^2$ be a general line through $p$.

\quad \emph{Claim 1:} We have $h^1(\Ii _B(d))=0$.

\quad \emph{Claim 2:} $B$ is the scheme-theoretic base locus of $|\Ii _B(d)|$.

To prove Claim 1 and Claim 2 we use induction on $d$. For both of them, the case $d=1$ is trivial.

\quad \emph{Proof of Claim 1:}  For the inductive proof use the
residual exact sequence
\begin{equation}\label{eqb2.1}
0 \to \Ii _{B\setminus B_{d-1}}(d-1) \to \Ii _B(d) \to \Ii _{B_{d-1},R_{d-1}}(d)\to 0.
\end{equation}

\quad \emph{Proof of Claim 2:} The scheme-theoretic base locus $\Bb$ of $|\Ii _B(d)|$ is contained in $R_0\cup \cdots \cup R_{d-1}$. Remark~\ref{b4.0}
gives $\Bb \cap R_{d-1} =B_{d-1}$. By the inductive assumption $B\setminus B_{d-1}$ is the base locus $\Bb '$ of $|\Ii _{B\setminus B_{d-1}}(d-1)|$
and $h^1(\Ii _{B\setminus B_{d-1}}(d-1)) =0$. Take $q\in R_0\cup \cdots \cup R_{d-1}$ such that $q\notin B$. If $q\in R_{d-1}$ we have
$h^1(R_{d-1},\Ii _{q\cup B_{d-1}, R_{d-1}}(d-1)) =0$ \cite[Corollary 3.13 pag 150]{miranda}. Since $h^1(\Ii _{B\setminus B_{d-1}}(d-1)) =0$, the residual exact sequence of $B\cup \{q\}$ with respect to $R_{d-1}$ gives $h^1(\Ii _{B\cup \{q\}}(d)) =0$, i.e. $q\notin \Bb$. Now assume $q\in R_i$ for some $i\le d-2$. The inductive assumption for Claim 2 gives
$q\notin \Bb'$ and so $h^1(\Ii _{(B\setminus B_{d-1})\cup \{q\}}(d-1)) =0$. The residual exact sequence of $B\cup \{q\}$ with respect to $R_{d-1}$ gives $h^1(\Ii _{B\cup \{q\}}(d)) =0$.
Thus $q\notin \Bb$.
Thus $B$ is the set-theoretic base locus of $|\Ii _B(d)|$. Since $\Bb \subset R_0\cup  \cdots \cup R_{d-1}$ and $B$ is contained in the smooth locus of $R_0\cup \cdots \cup R_{d-1}$, to prove Claim 2 it is sufficient to prove that for each $q\in B$, say $q\in R_i$,  $h^1(\Ii _{(B\setminus \{q\})\cup E}(d)) =0$, where $E$ is the degree $2$ zero-dimensional subscheme of $R_i$ with $q$ as its reduction. First assume $i=d-1$. We have $h^1(R_{d-1},\Ii _{(B_{d-1}\setminus \{q\})\cup E,R_{d-1}}(d)) =0$, because
$\deg ((B_{d-1}\setminus \{q\})\cup E) =d$ \cite[Corollary 3.13 pag 150]{miranda}. Since $ h^1(\Ii _{B\setminus B_{d-1}}(d-1))=0$, a residual exact sequence with respect to $R_{d-1}$ like~\eqref{eqb2.1} with $(B\setminus \{q\})\cup E$ instead
of $B$ gives $h^1(\Ii _{(B\setminus \{q\})\cup E}(d)) =0$. Now assume $i\le d-2$. Since $B\setminus B_{d-1} =\Bb '$ as schemes, we have
$h^1(\Ii _{((B\setminus B_{d-1})\setminus \{q\})\cup E}(d-1)) =0$. Use the residual exact sequence of $(B\setminus \{q\})\cup E$ with respect to $R_{d-1}$

If $x=1+(d-1)d/2$ we take $A:= B$. If $x<1+(d-1)d/2$ Remark~\ref{b4.0} and the proofs of Claims 1 and 2 show that we may take
as
$A$ any subset of
$B$ with cardinality $x$.
\end{proof}

\begin{proof}[Proof of Theorem~\ref{i2} case (2):]

 Let $E\subset \CC\PP^3$ be a general union of $\nu_n(d)$ twistor lines. Let $E'$ be a general union of $E$ and
$\nu(d)-\nu_n(d)$ twistor lines. By \textit{case (1)} we have $|\Ii _{E'}(d)|\ne \emptyset$ and every element of $|\Ii
_{E'}(d)|$ is integral.
Thus $|\Ii _E(d)|\ne \emptyset$ and a general $Y\in |\Ii _E(d)|$ is
irreducible. By Corollary~\ref{b3} we have
$h^1(\Ii _E(t)) =0$ for $t\ge d-1$ and in particular $\dim |\Ii _E(d)| =\binom{d+3}{3} -(d+1)\nu_n(d)-1$. Hence the thesis
follows from the following two claims.

\quad \emph{Claim 1:} A general $Y\in|\Ii_{E}(d)|$ is such that $E\cap Sing(Y)=\emptyset$. 

\quad \emph{Claim 2:}  A general $Y\in|\Ii_{E}(d)|$ is such that that $\mathrm{Sing}(Y)$ is finite.

\quad \emph{Proof of Claim 1:} Since $\dim E=1$, to prove that a general $Y \in |\Ii _E(d)|$ is smooth at
all points of
$E$ it is sufficient to prove that $h^0(\Ii _{2q\cup E}(d)) = h^0(\Ii _E(d)) -2$ for all $q\in E$,
i.e. we are going to prove that the space of surfaces singular at some point $q$ of $E$ has dimension strictly less than 
$|\Ii _E(d)|$ (recall Remark~\ref{r2.1}).
Since $q\in E$, the scheme
$2q\cap E$ has degree
$2$ and we have $h^0(\Ii _{2q\cup E}(d)) \le h^0(\Ii _E(d)) -2$. Hence to prove that the last inequality is an equality it is
sufficient
to prove that $h^1(\Ii _{2q\cup E}(d)) =0$. Call $L$ the connected component of $E$ containing $q$. Set $A:= E\setminus L$.

Let $H\subset \CC\PP^3$ be a general plane containing $L$ and set $S:= A\cap H$. Since any two lines of $H$ meet, $S$ is the union
of $\nu_n(d)-1$ distinct points. As
$\mathrm{Res}_H(2q) =\{q\}$ and
$q\in L$, we have $\mathrm{Res}_H(2q\cup E) = \{q\}\cup A$. Moreover, since $L\cap A=\emptyset$, we have $H\cap (2q\cup E)
= (2q,H)\cup L \cup S$. We have a residual exact sequence
\begin{equation}\label{eqb2}
0 \to \Ii _{\{q\} \cup A}(d-1) \to \Ii _{2q\cup E}(d)\to \Ii _{(2q,H) \cup L \cup S,H}(d)\to 0.
\end{equation}
Since $h^1(\Ii _E(d-1)) =0$, we have $h^1(\Ii _A(d-1))=0$ and $L$ imposes $d$ independent conditions to $|\Ii _A(d-1)|$. Since
$L\cong \CC\PP^1$
and $\Oo _L(d-1)$ is a degree $d-1$ line bundle, we get that any union of $d$ points of $L$ imposes $d$ independent conditions
to $|\Ii _A(d-1)|$. Hence $\dim |\Ii _{\{q\}\cup A}(d-1)| = \dim |\Ii _A(d-1)|-1$. Since $h^1(\Ii _A(d-1))=0$, we get
$h^1(\Ii _{\{q\}\cup A}(d-1)) =0$. By the long cohomology exact sequence of~\eqref{eqb2}, to conclude the proof that
a general $Y\in |\Ii _E(d)|$ is smooth at all points of $E$ it is sufficient to prove that $h^1(H,\Ii _{(2q,H)\cup L \cup
S,H}(d))=0$. We have $L\cap S=\emptyset$ and $L\cap (2q,H) =(2q,L)$ and so $\mathrm{Res}_{L,H}((2q,H)\cup S) =\{q\}\cup S$.
Hence we have 
$$h^1(H,\Ii _{(2q,H) \cup L \cup S,H}(d))=h^1(H,\Ii _{\{q\}\cup S}(d-1)).$$ 
Since $q\in L$ and $d\ge 2$, we have 
$$h^1(H,\Ii
_{\{q\}\cup S}(d-1)) \le h^1(H,\Ii _{L\cup S}(d-1)) = h^1(H,\Ii _{S,H}(d-2)).$$ 
The plane $H$ depends only from $L$, not from
the choice of $q$. Thus for a general $E$, we may see $S$ as the intersection of $H$ with $\nu_n(d)-1 = \nu(d-1)-1$
general twistor lines. Thus $S$ has the same cohomological properties of a general subset of $H$ with cardinality $\nu (d-1)-1$. Since $ \nu(d-1)-1 \le \binom{d}{2}$, we have $h^1(H,\Ii _{S,H}(d-2))=0$
and so $h^{1}( \Ii _{2q\cup E}(d))=0$ and a general $Y$ is smooth at $E$.

\quad \emph{Proof of Claim 2:} By Bertini's theorem \cite[page 137]{GH}, \cite[II.8.18]{h} and the
generality of $Y$ the set $\mathrm{Sing}(Y)$ is contained in the base locus $\Bb$ of $|\Ii _E(d)|$. Let $B_1,\dots ,B_s$,
$s\ge 0$, the positive dimensional components of $\Bb$ which are not one of the lines of $E$. Fix a general plane $H\subset
\CC\PP^3$
and set $S:= E\cap H$. Since $E$ is a general union of $\nu_n(d)$ twistor lines and (for this part of the proof) we may
first choose
$H$ independent from $E$, we may see $S$ as a general subset of $H$ with cardinality $\nu_n(d)$. By Lemma~\ref{b4} for each
$p\in H\setminus S$
we have $h^1(H,\Ii _{\{p\}\cup S,H}(d)) =0$. 
Since we saw that $h^1(\Ii _E(d-1)) =0$, the residual exact sequence of $H$ for the
sheaf $\Ii _{\{p\}\cup E}(d)$ 
$$0\to \Ii _E(d-1) \to \Ii _{\{p\}\cup E}(d) \to \Ii _{\{p\}\cup A,H}(d) \to 0,$$
gives $h^1( \Ii _{E\cup \{p\}}(d))=0$ and so
$p$ is not in the base locus of $|\Ii _E(d)|$. Thus either $s=0$ (and so $\mathrm{Sing}(Y)$ is finite) or each $B_i$ is an
irreducible curve meeting $E$. Since $Y$ is smooth at each point of $E$, each $B_i$ contains at most finitely many points of
$\mathrm{Sing}(Y)$. Thus $\mathrm{Sing}(Y)$ is finite. 

\end{proof}

To prove the third part of Theorem~\ref{i2}, we need the following lemma.

\begin{lemma}\label{b5}
Let $T\subset \CC\PP^3$ be any smooth quadric. Fix integers $b\ge a>0$ and $x$ with $0\le x \le (a+1)(b-1)$. Let $S\subset T$ be a
general union of $x$ points. Then $h^1(T,\Ii _{S,T}(a,b)) =0$ and $S$ is the scheme-theoretic base locus of $|\Ii _{S,T}(a,b)|$.
\end{lemma}

\begin{proof}
Since $S$ is general and $|S|=x=(a+1)(b-1)\le (a+1)(b+1) = h^0(T,\Oo _T(a,b))$, we
have $h^1(T,\Ii _{S,T}(a,b)) =0$. It is then sufficient to find a finite set $A\subset \CC\PP^2$ such that $|A|=x$,
 $h^1(\Ii _{A,T}(d)) =0$ and $A$ is the scheme-theoretic base locus of $|\Ii _{A,T}(d)|$. Take $a+1$ distinct elements $R_1,\dots
,R_{a+1}$ of $|\Oo _T(1,0)|$ and any $B_i\subset R_i$ with $|B_i| = b-1$. Set $B:= \cup _{i=1}^{a+1} B_i$. The divisor $R_{a+1}$ of
$T$ gives the residual exact sequence on $T$:
\begin{equation}\label{eqb2.2}
0 \to \Ii _{B\setminus B_{a+1}}(a-1,b) \to \Ii _{B,T}(a,b) \to \Ii _{B_1,R_{a+1}}(a,b)\to 0
\end{equation}

Note that $\Oo _{R_1}(a,b)$ is the degree $b$ line bundle on $R_1\cong \CC\PP^1$. Following the proof of Lemma~\ref{b4} we get Claims
1 and 2, i.e. we prove that we may take $A=B$ when $x=(a+1)(b-1)$. If $x< (a+1)(b-1)$ the same inductive proof using~\eqref{eqb2.2}
shows that we may take as $A$ any subset of $B$ with cardinality $x$.
\end{proof}

In the last step of the proof of Theorem~\ref{i2},we prove that $\nu_s (d)$ general twistor lines are contained in a smooth degree $d$ surface of
$\CC\PP^3$.

\begin{proof}[Proof of Theorem~\ref{i2} case (3):]
Thanks to the results contained in~\cite{hh} and the Density Lemma~\ref{b2}  we have $h^1(\Ii _E(d))=0$ for a general union
of $\nu_s (d)$ lines. By Bertini's theorem it is sufficient to prove that $|\Ii _E(d)|$ has no base points outside
$E$ and that a general $Y\in |\Ii_E(d)|$ is smooth at all points of $E$. The latter condition is satisfied thanks to the Proof of Theorem~\ref{i2} case (2), because $\nu_s (d)\le \nu_n(d)$. We saw that it is sufficient to check separately these two Zariski open properties.

Since $\nu_s (d) \le d$ for all $d\le 7$, by Theorem~\ref{i1} we could assume $d\ge 8$. However, we want to prove  by
induction on $d$ the following statement:

\quad \emph{Statement $\star _d$, $d\ge 0$:} For any general union $F\subset \CC\PP^3$ of $\nu_s (d)$ twistor lines the linear
system $|\Ii _F(d)|$ has no base point outside $F$. 

When $\nu_s (d)\le d$ (and in particular if $d \le 7$) we proved $\star _d$ in the proof of Theorem~\ref{i1}
(the part studying $\dim \Sigma _q$ when $q\in \CC\PP^3\setminus E$).  

Thus we may assume $d \ge 8$ and assume $\star _{d-2}$ is true, i.e. assume that $|\Ii _G(d-2)|$ has no base points in $\CC\PP^3\setminus G$ for a general union $G$ of
$\nu_s  (d-2)$ twistor lines.  Let $T\subset \CC\PP^3$ be a smooth $j$-invariant quadric in the sense of~\cite{sv1} (also called \textit{real}), i.e. assume
that one of the two rulings of $T$, say the ruling $|\Oo _T(1,0)|$, is formed by twistor lines. Let $M\subset T$ be any
union of $x:= \nu_s (d)-\nu_s (d-2)$
distinct elements of $|\Oo _T(1,0)|$ and take a general union
$F\subset
\CC\PP^3$ of $\nu_s (d-2)$ twistor lines. By the inductive assumption $|\Ii _F(d-2)|$ has no base point outside $F$.
Since any two different twistor lines are disjoint, we have $M\cap F:= \emptyset$. Since we fixed $F$ after fixing $T$, we may
assume that each line of $F$ is transversal to $T$. Thus $|F\cap T| = 2\nu_s (d-2)$. Note that $F\cap T \subset T\setminus
M$. Set $E:= F\cup M$. 
Since $E$ is not a general union of $\nu_s (d)$ lines, we need to prove the following claim.

\quad \emph{Claim:} For the just defined set of lines $E$ it holds $h^1(\Ii _E(d))=0$.

\quad \emph{Proof of Claim:}  First of all, any two points of
$T$ not contained in a line of
$T$ are the intersection of
$T$ with a line of
$\CC\PP^3$. Thus the union
$D$ of $\nu_s (d-2)$ general lines of $\CC\PP^3$ has the property that $D\cap T$ is a general union of $2\nu_s (d-2)$ lines. Notice that  $D\cap T\subset T\setminus M$ and that  $2\nu_s (d-2) \le (d+1)(d+1-x)$.
Since
$D\cap T$ is general, we get $h^1(T,\Ii _{D\cap T,T}(d-x,d)) =0$. The residual exact sequence of $M$ in $T$ gives
$h^1(T,\Ii _{(D\cap T)\cup M,T}(d,d)) =0$. By the Density Lemma the same is true for a general union of $\nu_s (d-2)$ twistor
lines, i.e. $h^1(T,\Ii _{(F\cap T)\cup M,T}(d,d)) =0$. 
 Since $h^1(\Ii _F(d-2)) =0= h^1(T,\Ii _{(F\cap T)\cup M,T}(d,d))$, to prove that   $h^1(\Ii _E(d))=0$ it is sufficient to use
the residual exact sequence
\begin{equation*}
0 \to \Ii _F(d-2)\to \Ii _E(d)\to \Ii _{(F\cap T)\cup M,T}(d,d)\to 0.
\end{equation*}

Having proved the claim, we now check that $|\Ii _E(d)|$ has no base point outside $E$. Since  $h^1(\Ii _E(d))=0$, it would be sufficient to prove
that $\nu_s (d)$ twistor lines are contained in a smooth degree $d$ surface. Since
$|\Ii _F(d-2)|$ has no base point outside $F$, $|\Ii _E(d)|$ has no base point outside $F\cup T$. Thus it
is sufficient to prove that $h^1(\Ii _{\{p\}\cup E}(d))=0$ for each $p\in T\setminus ((F\cap T)\cup M)$. Fix $p\in T\setminus
((F\cap T)\cup M)$. The residual exact sequence of $T$ gives
\begin{equation}\label{eqb4}
0 \to \Ii _F(d-2)\to \Ii _{\{p\}\cup E}(d)\to \Ii _{(F\cap T)\cup M\cup \{p\},T}(d,d)\to 0.
\end{equation}
Since $h^1( \Ii _F(d-2))=0$, the long cohomology exact sequence of~\eqref{eqb4} shows that it is sufficient to prove the vanishing of $h^1(T,\Ii _{(F\cap T)\cup
M\cup \{p\},T}(d,d))$. Since $F\cap M  =\emptyset$ and $p\notin M$, we have 
$$h^1(T,\Ii _{(F\cap T)\cup
M\cup \{p\},T}(d,d))=h^1(T,\Ii _{(F\cap T)\cup
\{p\},T}(d-x,d)).$$
 Note that $2\nu_s (d-2) \le (d-x+1)(d-1)$. By Lemma~\ref{b5} we have $h^1(T,\Ii _{(D\cap T)\cup
\{p\},T}(d-x,d)) =0$ for a general union of $\nu_s (d-2)$ lines of $\CC\PP^3$. By the Density Lemma~\ref{b2} this is true for a general
union of $\nu_s (d-2)$ twistor lines, i.e. $h^1(T,\Ii _{(F\cap T)\cup
\{p\},T}(d-x,d))=0$. Therefore we have proved $\star _d$ and hence the Theorem.
\end{proof}

Having proved Theorem~\ref{i2}, we pass to Theorem~\ref{i2.1}.

\begin{proof}[Proof of Theorem~\ref{i2.1}:]
Fix an integer $k\le \nu_j (d)$. Let $E\subset \CC\PP^3$ be a general union of $k$ twistor lines. The case $k=0$ is trivial, because a general surface of degree $\ge 5$ contains no line. Thus we may assume $k>0$ and in particular $d\ge 10$. Let $E\subset \CC\PP^3$ be a general union of $k$ twistor lines. Since $k \le \nu_s (d-6) \le \nu(d-8)$, we have $h^1(\Ii _E(t)) =0$ for all $t\ge d-8$. By Theorem~\ref{i2} case (3)
for each $t\ge d-8$ a general $W\in |\Ii _E(t)|$ is smooth. Thus we only need to prove that a general $Y\in |\Ii _E(d)|$ contains no line $L$ with $L\cap E =\emptyset$.
For any line $L\subset \CC\PP^3$ set $A(L):= |\Ii _{E\cup L}(d)|$ and $a(L):= h^0(\Ii _E(d)) -h^0(\Ii _{E\cup L}(d))$. $A(L)$ is a linear subspace of the projective space  $|\Ii _E(d)|$
with codimension $a(L)$ and it parametrizes all degree $d$ surfaces containing $E\cup L$. By the semicontinuity theorem for cohomology \cite[III.12.8]{h} for any integer $i\ge 0$ the set $\omega (i)$ of all lines $L\subset \CC\PP^3$ such that $L\cap E=\emptyset$ and $a(L)=i$ is a locally closed subset of $Gr (2,4)$. We call $\tau _i$ its dimension, i.e. the maximal dimension of one of the irreducible components of $\omega (i)$. To prove that a general $Y\in |\Ii _E(d)|$ contains no line $L$ with $L\cap E=\emptyset$, it is sufficient
to prove that $\tau _i < i$ (i.e. $\omega (0) =\emptyset$, $\omega (1)$ is finite or empty, and so on). Since $\dim Gr (2,4) =4$, it is sufficient to prove that $\tau _i < i$ for $0\le i \le 4$.
Note that $\nu_j (d) < \nu(d)$ and in particular $(d+1)(k+1) \le \binom{d+3}{3}$. By Corollary~\ref{b3} we have $h^1(\Ii _{E\cup L}(d)) =0$ for a general $L\in Gr (2,4)$,
i.e. $a(L) =d+1$ for a general $L\in Gr (2,4)$. Thus $\tau _i \le 3$ for all $i\le d$. Thus it is sufficient to prove that $\tau _i <i$ for $i=0,1,2,3$.

Fix $4$ distinct quadric surfaces $Q_1$, $Q_2$, $Q_3$ and $Q_4$ such that $Q_1\cap Q_2\cap Q_3\cap Q_4=\emptyset$. We fix these
quadrics before fixing
$E$. Let
$\Ee$ be the set of all unions $F$ of $k$ distinct twistor lines with $h^1(\Ii _F(t)) =0$ for $t\ge d-6$ and intersecting
transversally each quadric $Q_j$, $j=1,2,3,4$. By the Density Lemma the set of all $F\cap Q_j$, $F\in \Ee$, is Zariski dense in
the
$2$-dimensional variety $Q_j[k]$ parametrizing all subsets of $Q_j$ with cardinality $2k$.  Note that $2k \le 2\nu_j (d) \le
(d-5)^2 -2(d-5)$. Again, by the Density Lemma for a general $E\in \Ee$ we have $h^1(Q_j,\Ii _{E\cap Q_j,Q_j}(a,b)) =0$ for all $a\ge
d-6$, $b\ge d-6$, $j=1,2,3,4$, and each linear system $|\Ii _{E\cap Q_j}(a,b)|$ has no base point. Fix a line $L\subset \CC\PP^3$
such that $L\cap E=\emptyset$. 
 

 First assume that $L$ is contained in one of the quadrics $Q_1,Q_2,Q_3,Q_4$, say in $Q_j$ as an element of $|\Oo _{Q_j}(0,1)|$. Since no line in $E$ is contained in $Q_j$, we have the following residual exact sequence
 \begin{equation}\label{eqb5}
 0 \to \Ii _E(d-2) \to \Ii _{E\cup L}(d) \to \Ii _{E\cap Q_j,Q_j}(d)\to 0 
 \end{equation}
 Since $E\in \Ee$ we have $h^1(\Ii _E(d-2)) =0$ and $h^1(Q_j,\Ii _{E\cap Q_j,Q_j}(d,d-1)) =0$. Since we have $h^1(Q_j, \Ii _{E\cap Q_j,Q_j}(d))=h^1(Q_j,\Ii _{E\cap Q_j,Q_j}(d,d-1))$,
\eqref{eqb5} gives $L\in \omega (d+1)$.

 Now assume $L\nsubseteq Q_j$ for any $j$. 
 In this case we want to prove that $ h^0(\Ii_{E\cup L}(d)) \le h^0(\Ii_E(d)) -4$, i.e. $L\notin \omega (i)$ for $i\le 3$.
 Fix $p_j\in L\cap Q_j$. To prove that $L\notin \omega (i)$ for $i\le 3$ (and so that $\tau_{i}<i$ for $i\le 3$) it is sufficient to prove that
 $h^0(\Ii _{E\cup \{p_1,p_2,p_3,p_4\}}(d)) =h^0(\Ii _E(d)) -4$, i.e. that $h^1(\Ii _{E\cup \{p_1,p_2,p_3,p_4\}}(d)) =0$. 
 We prove that this is the case
 if $p_j$ is not contained in $Q_i$ when $i\ne j$. With this assumption we have the $4$ residual exact sequences
  \begin{equation}\label{eqb9}
 0 \to \Ii _E(d-8) \to \Ii _{E\cup \{p_1\}}(d-6) \to \Ii _{(E\cap Q_1)\cup \{p_1\},Q_1}(d-6)\to 0
 \end{equation}
  \begin{equation}\label{eqb8}
 0 \to \Ii _{E\cup \{p_1\}}(d-6) \to \Ii _{E\cup \{p_1,p_2\}}(d-4) \to \Ii _{(E\cap Q_2)\cup \{p_2\},Q_3}(d-4)\to 0
 \end{equation}
  \begin{equation}\label{eqb7}
 0 \to \Ii _{E\cup \{p_1,p_2\}}(d-4) \to \Ii _{E\cup \{p_1,p_2,p_3\}}(d-2) \to \Ii _{(E\cap Q_3)\cup \{p_3\},Q_3}(d-2)\to 0
 \end{equation}
 \begin{equation}\label{eqb6}
 0 \to \Ii _{E\cup \{p_1,p_2,p_3\}}(d-2) \to \Ii _{E\cup \{p_1,p_2,p_3,p_4\}}(d) \to \Ii _{(E\cap Q_4)\cup \{p_4\},Q_4}(d)\to 0
 \end{equation}
 We now use repeatedly Lemma~\ref{b5} and its consequences.
\begin{itemize}
\item Since $h^1(\Ii _E(d-8)) =0$ and $ h^1(Q_1,\Ii _{(E\cap Q_1)\cup \{p_1\},Q_1}(d-6))=0$, \eqref{eqb9} gives $h^1(\Ii _{E\cup \{p_1\}}(d-6))=0$.
\item Since we have $h^1(Q_2,\Ii _{(E\cap Q_2)\cup \{p_2\},Q_3}(d-4))=0$, \eqref{eqb8} gives $h^1( \Ii _{E\cup \{p_1,p_2\}}(d-4))=0$.
\item  Since $h^1(Q_3,\Ii _{(E\cap Q_3)\cup \{p_3\},Q_3}(d-2))=0$, \eqref{eqb7} gives
$h^1(\Ii _{E\cup \{p_1,p_2,p_3\}}(d-2))=0$.
\item  Since we have $h^1(Q_4,\Ii _{(E\cap Q_4)\cup \{p_4\},Q_4}(d)) =0$, \eqref{eqb6} gives $h^1(\Ii _{E\cup \{p_1,p_2,p_3,p_4\}}(d)) =0$, concluding the proof in this case.
\end{itemize}

Assume now that $p_i\in Q_j$ for some $i\ne j$. First of all, to do the construction without any
modification
we only need $p_i\notin Q_j$ for all $i<j$ (we do not care if $Q_3$ contains $p_4$, if $Q_2$ contains $p_3$ or $p_2$ or
if $Q_1$ contains $p_2$ or $p_3$ or $p_4$). Instead
of fixing
$4$ smooth quadrics, we fix $10$ smooth quadrics
$Q[h]$, $1\le h\le 10$, such that any $4$ of these quadrics have no common point. We set $Q_1:= Q[1]$ and fix $p_1\in L\cap
Q_1$. Since any $4$ of the quadrics $Q[h]$ have empty intersection, there is $Q[h_1]$ such that $p_1\notin Q[h_1]$. We set
$Q_2:= Q[h_1]$ and take $p_2\in L\cap Q_2$. Since any  $4$ of the quadrics $Q[h]$ have empty intersection, there is $Q[h_2]$
with $p_1\notin Q[h_2]$ and $p_2\notin Q[h_2]$. Set $Q_3:= Q[h_2]$ and fix $p_3\in Q_3\cap L$. Since any  $4$ of the quadrics
$Q[h]$ have empty intersection, there is $Q[h_3]$ with $\{p_1,p_2,p_3\}\cap Q[h_3] =\emptyset$. Set $Q_4:= Q[h_3]$ and
take any $p_4\in L\cap Q_4$.
Using again the same argument as before for these new four quadrics, we conclude the proof.
\end{proof}


We now pass to the last theorem. To precise what stated in the introduction we
set the following notation. Let $z\in\CC^{4}$. For any $\alpha=(\alpha_{0},\alpha_{1},\alpha_{2},\alpha_{3})\in\NN^{4}$, set
$$z^{\alpha}= z_{0}^{\alpha_{0}}z_{1}^{\alpha_{1}}z_{2}^{\alpha_{2}}z_{3}^{\alpha_{3}}.$$
For any $\alpha\in\NN^{4}$ we write $|\alpha|=\alpha_{0}+\alpha_{1}+\alpha_{2}+\alpha_{3}$.
We now extend the map $j$ to the $\binom{d+3}{3}$-dimensional vector space $H^0(\Oo _{\CC\PP^3}(d))$. 
%
Consider the map $j: H^0(\Oo _{\CC\PP^3}(d)) \to H^0(\Oo _{\CC\PP^3}(d))$
defined in the following way
$$
H^0(\Oo _{\CC\PP^3}(d)) \ni f = \sum _{|\alpha|=d} c_{\alpha} z^{\alpha}\mapsto j(f) = \sum _{|\alpha|=d}\hat{c}_{\alpha}z^{\alpha},
$$
where $\hat{c}_{\alpha _0,\alpha _1,\alpha _2,\alpha _3} = (-1)^{\alpha_{0}+\alpha_{2}}\bar{c}_{\alpha _1,\alpha _0,\alpha _3,\alpha _2}$.

Note that $j$ is $\RR$-linear and that  $j^2(f) = (-1)^d(f)$. For any $z\in \CC^4$ we
have
$f(j(z)) =(-1)^d\overline{j(f)(z)}$. We say that $f = \sum _\alpha c_{\alpha} z^{\alpha}$ is
$j$-invariant if and only if there is a constant $a\in \CC\setminus \{0\}$ such that $j(f)=af$
i.e.: if $j$ fix the line $\CC f$.

By the Density Lemma, the set of all
elements of $H^0(\Oo _{\CC\PP^3}(d))$ associated to a fixed $a\in  \CC\setminus \{0\}$ is a real vector space of dimension $\binom{d+3}{3}$
over $\RR$.
Hence 
$j$ induces a bijection 
between the $\binom{d+3}{3}-1$ complex projective space $|\Oo _{\CC\PP^3}(d)|$ and itself. 

\begin{definition}
 We say that a surface $Y\in |\Oo _{\CC\PP^3}(d)|$ is $j$-\textit{invariant} if
 $j(z)\in Y$ for every $z\in Y$, i.e.: if $j(Y) = Y$.
\end{definition}

In some paper (e.g.~\cite{sv1}), a $j$-invariant surface $Y$ is called \textit{real} with respect to $j$.

We pass now to the proof of Theorem~\ref{i3}.
More precisely, for a general union $E\subset \CC\PP^3$ of $k$ twistor line we will prove the existence of a smooth $Y\in |\Ii _E(d)|$ such that $j(Y)=Y$ and $Y$ contains no line $L\subset \CC\PP^3$ with $L\cap E=\emptyset$ and in particular it contains no other twistor line.

\begin{proof}[Proof of Theorem~\ref{i3}] By Corollary~\ref{b3} we have
$h^1(\Ii _E(d)) =0$ and so $h^0(\Ii _E(d)) =\binom{d+3}{3} -k(d+1)$. 
The integer $\binom{d+3}{3} = (d+3)(d+2)(d+1)/6$ is even if and only if either $d$ is odd or $d\equiv 2\pmod{4}$.
The integer $d+1$ is odd if and only if $d$ is even, thus $\binom{d+3}{3} -k(d+1)$ is odd, i.e. $\dim |\Ii _E(d)|$ is even, if and only if either $d\equiv 2\pmod{4}$ and $k$ is odd or 
$d\equiv 0\pmod{4}$ and $k$ is even, i.e. in the cases needed to prove Theorem~\ref{i3}.

Fix any union $E\subset \CC\PP^3$ of finitely many twistor fibers. For each $t\in \ZZ$ the projective space $|\Ii _E(t)|$ is $j$-invariant.

\quad \emph{Claim:} If $\dim |\Ii _E(t)|$ is even, then the set of all $j$-invariant element of $|\Ii _E(t)|$ is Zariski dense in $|\Ii _E(t)|$.

\quad \emph{Proof of Claim:} First assume $h^0(\Ii _E(t)) =1$, i.e. assume $|\Ii _E(t)| =\{Y\}$ for some $Y$. Since $j(E)=E$, w have $\{j(Y)\} = \{Y\}$ and hence $j(Y)=Y$.
Now assume $h^0(\Ii _E(t)) >1$ and $h^0(\Ii _E(t))$ odd, say $h^0(\Ii _E(t)) =2m+1$ for some positive integer $m$. Fix a general set $S\subset \CC\PP^3$ such that $|S|=m$.
Since $S$ is general, we have $h^0(\Ii _{E\cup S\cup j(S)}(t)) =1$. Since $j(E\cup S\cup j(S)) = E\cup S\cup j(S)$, the proof just given shows that the only element
$A\in |\Ii _{E\cup S\cup j(S)}(t)|$ satisfies $j(Y) =Y$.
By our assumptions on $d$ and $k$ there is $Y\in |\Ii _E(d)|$ such that $j(Y)=Y$. Thus  $j$ induces an anti-holomorphic involution on the projective space
$|\Ii _E(d)|$ with at least one fixed point. Thus the set $\Gamma$ of all $Y\in |\Ii _E(d)|$ with $j(Y)=Y$ is Zariski dense in $|\Ii _E(d)|$.

Hence (by the proof of Theorem~\ref{i2.1})
there is a smooth $Y\in \Gamma$ containing no line $L$ with $L\cap E=\emptyset$.
\end{proof}

\end{document}